\documentclass[11pt,reqno]{amsart}
\usepackage{filecontents}
\usepackage{microtype}
\usepackage[english]{babel}
\usepackage[T1]{fontenc}
\usepackage[utf8]{inputenc}
\usepackage[T1]{fontenc}
\usepackage{lmodern}
\usepackage{latexsym}
\usepackage{verbatim}
\usepackage{hyperref}
\usepackage{amsthm}
\usepackage{amsmath}
\usepackage{amsfonts}
\usepackage[italian]{varioref}
\usepackage{imakeidx}
\usepackage{amssymb}
\usepackage{xcolor}
\usepackage{esint}
\usepackage{accents}
\usepackage{enumitem}
\usepackage[a4paper,top=4cm,bottom=4cm,left=3.5cm,right=3.5cm]{geometry}
\usepackage{caption} 

\usepackage{tikz}
\usepackage{tikz,tikz-3dplot}
\usepackage{pgfplots}
\pgfplotsset{compat=1.16}
\tikzset{
    cross/.pic = {
    \draw[rotate = 45] (-#1,0) -- (#1,0);
    \draw[rotate = 45] (0,-#1) -- (0, #1);
    }
}
\usetikzlibrary{calc, angles}
\usetikzlibrary{arrows}
\usetikzlibrary{positioning}
\usetikzlibrary{intersections}
\usepackage{pgfplots}
\pgfplotsset{compat=newest}
\usepgfplotslibrary{fillbetween}
\usepackage{pgfplots}
\pgfplotsset{compat=newest}
\usepgfplotslibrary{fillbetween}
\usetikzlibrary{patterns}

	\definecolor{ao(english)}{rgb}{0.0, 0.5, 0.0}

\usepackage{lipsum}
\usepackage{curve2e}
\definecolor{gray}{gray}{0.4}

\usepackage{amsmath}
\usepackage{amsfonts}
\usepackage{amssymb}
\usepackage{graphicx}
\usepackage{hyperref} 
\usepackage{mathrsfs} 
\usepackage{relsize} 

\theoremstyle{plain}
\newtheorem{theorem}{Theorem}[section]

\newtheorem{definition}[theorem]{Definition}
\newtheorem{lemma}[theorem]{Lemma}
\newtheorem*{problem*}{Problem}

\newtheorem{proposition}[theorem]{Proposition}

\theoremstyle{definition}

\newtheorem{remark}[theorem]{Remark}

\def\Item$#1${\item $\displaystyle#1$
   \hfill\refstepcounter{equation}(\theequation)}
   
\theoremstyle{plain} 

\theoremstyle{definition}

\theoremstyle{remark} 
\numberwithin{equation}{section}
\usepackage[makeroom]{cancel}

\usepackage{bm} 
\usepackage{graphicx}
\makeindex
\usepackage{mathtools}
\usepackage{slashed}

\usepackage{tikz,tikz-3dplot}
\tdplotsetmaincoords{80}{45}
\tdplotsetrotatedcoords{-90}{180}{-90}
\usetikzlibrary{arrows, automata, intersections}

\newcommand{\rr}{\mathbb{R}}

\newcommand{\e}{\epsilon}

\newcommand{\inn}{\textnormal{in}\ }

\newcommand{\andd}{\textnormal{and}}

\newcommand{\bra}[1]{\{#1\}}
\newcommand{\loc}{_{loc}}

\newcommand{\system}[2]{\left\{\begin{array}{#1} #2 \end{array} \right.}

\newcommand{\dint}[1]{\ \textnormal{d}#1}

\newcommand{\norm}[2]{\left|\left| #1 \right| \right|_{#2}}

\newcommand{\dvol}{\textnormal{dv}}

\tikzstyle{mybox} = [draw=black, very thick, rectangle, rounded corners, inner ysep=5pt, inner xsep=5pt]

\begin{document}

\title[$L^p$ positivity preservation and self-adjointness]{$L^p$ positivity preservation and self-adjointness on incomplete Riemannian manifolds}
\author{Andrea Bisterzo and Giona Veronelli}
\address{Universit\`a degli Studi di Milano-Bicocca\\ Dipartimento di Matematica e Applicazioni \\ Via Cozzi 55, 20126 Milano - ITALY}
\email{a.bisterzo@campus.unimib.it}

\address{Universit\`a degli Studi di Milano-Bicocca\\ Dipartimento di Matematica e Applicazioni \\ Via Cozzi 55, 20126 Milano - ITALY}
\email{giona.veronelli@unimib.it}

\date{}

\maketitle


\begin{abstract}
The aim of this paper is to prove a qualitative property, namely the \textit{preservation of positivity}, for Schr\"odinger-type operators acting on $L^p$ functions defined on (possibly incomplete) Riemannian manifolds. A key assumption is a control of the behaviour of the potential of the operator near the Cauchy boundary of the manifolds. As a by-product, we establish the essential self-adjointness of such operators, as well as its generalization to the case $p\neq 2$, i.e. the fact that smooth compactly supported functions are an operator core for the Schr\"odinger operator in $L^p$.
\end{abstract}

\section{Introduction}

Let $(M,g)$ be a Riemannian manifold, $\Delta$ its negative definite Laplace-Beltrami operator (so that $\Delta = \tfrac{\partial^2}{\partial x^2}$ on $\rr$) and $V\in L^1\loc(M)$. Given a family of functions $\mathcal{S}\subseteq L^1\loc(M)$, we say that the $\mathcal{S}$ \textit{positivity preserving property} holds in $M$ for the operator $-\Delta+V$ if for every $u\in \mathcal{S}$
\begin{align*}
(-\Delta+V)u\geq 0 \quad \Rightarrow \quad u\geq 0\ \textnormal{a.e.},
\end{align*}
where the first inequality is understood in the sense of distributions. Recall that a function $u\in L^1\loc(M)$ satisfies $(-\Delta+V)u\geq 0$ (resp. $\leq 0$) \textit{in the sense of distributions} if
\begin{align*}
\int_M u(-\Delta+V)\psi\, \dvol \geq 0 \quad (\textnormal{resp.}\ \leq 0)
\end{align*}
for every $0\leq \psi\in C_c^\infty(M)$.

The positivity preservation for Schr\"odinger operators has been extensively studied in recent years. This definition was introduced by B. G\"uneysu in the paper \cite{Gu1} for the differential operator $-\Delta +1$, although the property first appeared in \cite{Ka} and \cite{BMS}. In particular, in \cite{BMS} the authors proved that the $L^2$ positivity preserving property for $-\Delta +1$ implies the essential self-adjointness of any operator $-\Delta+V$ with $0\leq V \in L^2\loc$. Since this type of operators are known to be essentially self-adjoints on complete manifolds (see \cite{BMS,Sh}), M. Braverman, O. Milatovic and M. Shubin conjectured that the differential operator $(-\Delta+1)$ must satisfy the $L^2$ positivity preservation on every complete Riemannian manifold. This assertion has been popularized under the name of \textsl{BMS conjecture} from the names of the three authors, \cite{Gu3}.

%

The BMS conjecture has been addressed by several authors, possibly considering additional assumptions on the geometry of the manifold at hand. See, for instance, \cite{BMS,BS,Gu1,Gu2,Ka0,MV}. Recently, it has been proved in the positive by S. Pigola, D. Valtorta and the second author in \cite{PVV} (see also \cite{GPSV} for a generalization to non-smooth Dirichlet spaces). Using a monotonic approximation argument and some regularity results for subharmonic distributions, they proved that on every complete Riemannian manifold the operator $-\Delta+1$ satisfies the $L^p$ positivity preserving property for any $p\in(1,+\infty)$.

Regarding the cases $p=1$ and $p=+\infty$, without further assumptions the $L^p$-positivity preservation property  for the operator $-\Delta+1$ in general might fail even for complete manifolds. In this respect, in the recent work \cite{BM} the first author and L. Marini determined
\begin{itemize}
\item that the $L^\infty$ positivity preserving property is equivalent to the stochastic completeness of the manifold (and thus unrelated to the geodesic completeness);

\item the optimality of Theorem II in \cite{MV}, which states that geodesic completeness and $\textnormal{Ric}(x)\geq -C r^2(x)$ outside a compact set imply the $L^1$ positivity preservation for the operator.
\end{itemize}\smallskip

The $L^p$-positivity preserving property for more general Schr\"odinger operators $-\Delta + V$ acting on complete Riemannian manifolds has been considered independently in the recent works \cite{ACR} and \cite{BFP}. In particular, in this latter preprint the authors established on the one hand the positivity preserving property for a class of $L^p\loc$ functions whose $L^p$ norm over geodesic balls satisfies a certain growth condition. On the other hand, for $p\in (1,+\infty)$, they successfully dealt with differential operators of the form $-\Delta+V$, where $0\leq V\in L^1\loc(M)$ may decays to $0$ at infinity. 

In a different direction, in \cite{PVV} the authors also managed to prove that the $L^p$ positivity preservation for the operator $-\Delta+1$ is stable by removing from a complete manifold a possibly singular subset satisfying certain Hausdorff co-dimension or uniform Minkowski-type conditions. 
As a consequence, they showed the essential self-adjointness in $L^2$ of Schr\"odinger operators of the form $-\Delta + V$ for lower bounded potential $V\in L^2_{loc}$, as well as the analogous spectral counterparts in $L^p$, that is the fact that $C^\infty_c$ is an operator core for the Schr\"odinger operator; see Section \ref{Sec:Core} for more details.\\\smallskip

The search for sufficient conditions to the validity of the essential self-adjointness of Schr\"odinger operators has been widely studied over the years. In the case of complete manifolds, let us mention at least \cite{Gu1,GPo,Ka2,Mi1,Mi2,Sh,Sh0}. In the incomplete case, the essential self-adjointness of the Laplace-Beltrami operator (i.e. when $V=0$) was first investigated by Y. Colin de Verdière and J. Masamune \cite{CdV,Ma} when the singularity has integer codimension, and by M. Hinz, J. Masamune and K. Suzuki in the recent \cite{HMS} in which the removed compact set may have non-integer codimension.
 
On the other hand, in \cite{MT} O. Milatovic and F. Truc adopted a different point of view in the study of the essential self-adjointness in $L^2$ of Schr\"odinger operators of the form $-\Delta + V$. Namely, they considered geodesically incomplete manifolds without requiring any assumption on the geometry of $M$, and in particular on the codimension of the Cauchy boundary. The price to pay is a much stronger restriction on the potential $V$, which is required to explose at least quadratically near the Cauchy boundary of $M$; see \cite[Theorem 3]{MT} for a precise statement.
In view of the results alluded to above, it is thus natural to speculate that an intermediate control on the behavior of the potential near the boundary can be combined with an intermediate bound on the codimension of the Cauchy boundary, to get assumptions which, in a sense, interpolate between the ones in 
\cite{MT} and in \cite{PVV}. This is the content of Theorem \ref{Thm:ESA}. 

As explained above, the approach to the essential self-adjointness through the $L^2$-positivity preservation that we adopted here naturally generalizes to the $L^p$ setting. We formalize this abstract phenomenon in Theorem \ref{Thm:Final}. As a concrete instance,
in Theorem \ref{Thm:Core} we prove that $C^\infty_c$ is an operator core in $L^p$ for the Schr\"odinger operator $-\Delta +V$ under assumptions on the smallness of the Minkowski dimension of the Cauchy boundary and on the growth of the potential $V$ near the boundary.


\medskip

The paper is organized as follows. In Section \ref{Sec:LpPP} we prove the $L^p$ positivity preserving property for operators of the form $-\Delta+V$, with $0\leq V\in L^\infty\loc$ having suitable lower bounds, which act on (possibly incomplete) Riemannian manifolds whose Cauchy boundary satisfies a Minkowski-type condition; see Theorem \ref{Thm:Lp}. Following the heuristic described above, in Sections \ref{Sec:ESA} and \ref{Sec:Core} we apply the positivity preserving property just obtained in order to prove respectively that this class of operators on $C^\infty_c$ are essentially self-adjoint in $L^2$ and that $C^\infty_c$ is an operator core for the maximal $p$-extension of $-\Delta+V$. These results are stated in Theorem \ref{Thm:ESA} and in Theorem \ref{Thm:Core}.



\section{$L^p$ positivity preservation}\label{Sec:LpPP}

This section is aimed at proving the following

\begin{theorem}\label{Thm:Lp}
Let $(N,h)$ be a complete Riemannian manifold and define $M:=N\setminus K$, where $K\subset N$ is a compact subset. Consider $V\in L^\infty\loc (M)$ so that
\begin{align*}
V(x)\geq \frac{C}{r^m(x)} \quad in\ M,
\end{align*}
where $C\in [0,1]$ and $m\in  \{0,2\}$ are positive constants and $r(x):=d^N(x,K)$ is the distance function from $K$. Fix $p\in (1,+\infty)$.

If there exist two positive constants $E\geq 1$ and
\begin{align}\label{Cond:1p}
h\geq\system{ll}{0 & if\ m=2\ and\ C=\frac{1}{p-1}\\  \frac{p+p\sqrt{1-(p-1)C}}{p-1} & if\ m=2\ and\ C\in \left(0,\frac{1}{p-1}\right)\\  \frac{2p}{p-1} & if\ m=0}
\end{align}
so that
\begin{align}\label{Cond:2p}
|B_r(K)|\leq E r^h \quad as\ r\to 0,
\end{align}
then the differential operator $-\Delta+V$ has the $L^p$ positivity preserving property.
\end{theorem}

 \begin{remark}
Reasoning as in \cite[Section 5 ]{PVV}, it is easy to see that Theorem \ref{Thm:Lp}, and consequently Theorems \ref{Thm:ESA} and \ref{Thm:Core}, hold as well if $N$ is assumed to be $q$-parabolic for some $q\ge \frac{2p}{p-1}$, but possibly incomplete.  
\end{remark}
\begin{remark}
As explained in the introduction, the case $m=0$ recovers a result obtained in \cite{PVV}.
\end{remark}
\medskip

\subsection{Preliminary results}
In order to prove Theorem \ref{Thm:Lp} we need two fundamental tools. The first is the classical Brezis-Kato inequality. We refer to \cite{Br,Po} for the Euclidean result and to \cite{PVV} for the Riemannian version.

\begin{proposition}[Brezis-Kato inequality]\label{Prop:BrezisKato}
Let $(M,g)$ be a Riemannian manifold and $V$ a measurable function over $M$.

If $u\in L^1\loc(M)$ is so that $Vu\in L^1\loc(M)$ and satisfies $(-\Delta+V)u\leq 0$ in the sense of distributions, then
\begin{align*}
(-\Delta+V)u^+ \leq 0 \quad in\ the\ sense\ of\ distributions,
\end{align*}
where $u^+(x):=\max\bra{u(x),0}$.
\end{proposition}

The second ingredient is the regularity result contained in \cite[Proposition 2.2]{BFP}. Initially stated for complete Riemannian manifolds, we stress that its original proof recovers in fact also the case of incomplete Riemannian manifolds. Before stating this result, we recall that the negative part of a real-valued function, denoted with $u^-$, is defined as
\begin{align*}
u^-(x):=\max \bra{-u,0} = (-u)^+(x).
\end{align*}
Using the above notation, the mentioned regularity result states what follows.

\begin{proposition}\label{Prop:Caccioppoli1}
Let $(M,g)$ be a (possibly incomplete) Riemannian manifolds and $0\leq V\in L^\infty\loc(M)$.

If $u\in L^1\loc(M)$ satisfies $(-\Delta+V)u\geq 0$ in the sense of distributions, then
\begin{enumerate}
\item $u^- \in L^\infty\loc(M)$ and $(u^-)^{p/2}\in W^{1,2}\loc(M)$ for every $p\in (1,+\infty)$;
\item for every $p\in (1,+\infty)$ the function $u^-$ satisfies
\begin{align}\label{Eq:Caccioppoli}
(p-1) \int_M V(u^-)^p \varphi^2 \dvol \leq \int_M (u^-)^p |\nabla \varphi|^2 \dvol
\end{align}
for every $0\leq \varphi \in C_c^{0,1}(M)$.
\end{enumerate}
\end{proposition}
\medskip

\subsection{Positivity preserving}
In this subsection, we will prove the positivity preserving property stated as Theorem \ref{Thm:Lp}, which is based on the inequality \eqref{Eq:Caccioppoli}. To this aim, let $R>\e>2\eta> 0$ and $\delta>0$ and consider the following real function $\psi:\rr_{\geq 0}\to \rr_{\geq 0}$ 
\bigskip

\hspace{-0.5cm}\begin{minipage}{5cm}
\begin{align*}
\psi_{R,\e,\eta}(t):=\system{ll}{0 & \inn [0, \eta) \\ \frac{t-\eta}{\eta} \left( \frac{2\eta}{\e}\right)^\delta  & \inn [\eta, 2\eta)\\ \left(\frac{t}{\e}\right)^\delta & \inn [2\eta, \e) \\ 1 & \inn [\e, R) \\ \frac{R+\eta-t}{\eta} & \inn  [R,R+\eta) \\ 0 & \inn [R+\eta,+\infty).}
\end{align*}
\end{minipage}
\hfill\hspace{-0.5cm}
\begin{minipage}{5cm}
\begin{center}
\begin{tikzpicture}[scale=1.2]
\path[draw,->] (-0.1,0)--(4,0);
\path[draw,->] (0,0)--(0,1.5);
\path[draw] (-0.05,1)--(0.05,1);
\path (-0.05,1) node[anchor=east]{$1$};
\path[red, draw, line width=0.5mm] (0,0)--(0.2,0)--(0.4,0.62);
\path[red,draw, line width=0.5mm] (0.4,0.6) arc(180:90: 0.6 and 0.4);
\path[red, draw, line width=0.5mm](1,1)--(2,1)--(2.5,0)--(3.96,0);
\path[black] (1,1.3) node {$\psi_{R,\e,\eta}$};
\path (-0.2,0) node[anchor=north] {$0$};
\path (0.1,0) node[anchor=north] {$\eta$};
\path[dashed, draw] (0.4,0)--(0.4,0.58) (0.5,0) node[anchor=north] {$2\eta$};
\path[dashed, draw] (1,0)--(1,1) (1,0) node[anchor=north] {$\e$};
\path[dashed, draw] (2,0)--(2,1) (2,0) node[anchor=north] {$R$};
\path[dashed, draw] (3,0) node[anchor=north] {$R+\eta$};
\end{tikzpicture}
\end{center}
\end{minipage}
\bigskip

%

Let $(N,g)$ be a complete Riemannian manifold and define $M:=N\setminus K$, where $K\subset N$ is a compact subset. Denote with $r(x):=d^N(x,K)$ the distance function from $K$ and consider the following cut-off function
\begin{align*}
\varphi_{R,\e,\eta}:=\left(\psi_{R,\e,\eta}\circ r\right)\in C^{0,1}_c(M).
\end{align*}
In particular, $\varphi_{R,\e,\eta}$ can be extended to $0$ in $K$, obtaining $\varphi_{R,\e,\eta}\in C^{0,1}_c(N)$.

We are now in a position to prove the $L^p$ positivity preserving property.

\begin{proof}[Proof of Theorem \ref{Thm:Lp}]
Let $v\in L^p$ be a solution to $(-\Delta+V)v\geq 0$ and denote $u:=v^- \geq 0$. Fix $\delta>0$ and for $0\leq2\eta<\e<R$ consider the function $\varphi_{R,\e,\eta}$.

\textbf{Step 1.} We start by supposing that the support of $v$ is compact in $N$. Fix $s\in(1,p]$. By applying \eqref{Eq:Caccioppoli} to the test functions $\varphi_{R,\e,\eta}$, we get
\begin{align*}
(s-1) \int_M u^s V \varphi_{R,\e,\eta}^2 \dvol \leq \int_M u^s |\nabla \varphi_{R,\e,\eta}|^2 \dvol.
\end{align*}
On the one hand, we have
\begin{align*}
(s-1) & \int_M u^s V \varphi_{R,\e,\eta}^2 \dvol \\
& \geq (s-1) \int_{B_\e \setminus B_{2\eta}} u^s \frac{C}{r^m} \left(\frac{r}{\e} \right)^{2\delta} \dvol + (s-1) \int_{B_R\setminus B_\e} u^s V \dvol
\end{align*}
while, on the other hand, choosing $R$ big enough so that the support of $u$ is contained in $B_R$,
\begin{align*}
\int_M u^s |\nabla \varphi_{R,\e,\eta}|^2 \dvol \leq \int_{B_{2\eta}\setminus B_\eta} u^s \frac{1}{\eta^2} \left(\frac{2\eta}{\e} \right)^{2\delta} \dvol + \int_{B_\e \setminus B_{2\eta}} u^s \delta^2 \frac{r^{2\delta-2}}{\e^{2\delta}} \dvol.
\end{align*}
By putting together previous inequalities, we obtain
\begin{align*}
 (s-1) & \int_{B_R\setminus B_\e} u^s V \dvol\\
& \leq \int_{B_{2\eta}\setminus B_\eta} u^s \frac{1}{\eta^2} \left(\frac{2\eta}{\e} \right)^{2\delta} \dvol \\
& \quad \quad \quad + \int_{B_\e \setminus B_{2\eta}} u^s \delta^2 \frac{r^{2\delta-2}}{\e^{2\delta}} \dvol - (s-1) \int_{B_\e \setminus B_{2\eta}} u^s \frac{C}{r^m} \left(\frac{r}{\e} \right)^{2\delta} \dvol\\
& = \int_{B_{2\eta}\setminus B_\eta} u^s \frac{1}{\eta^2} \left(\frac{2\eta}{\e} \right)^{2\delta} \dvol + \int_{B_\e \setminus B_{2\eta}} u^s \frac{r^{2\delta-2}}{\e^{2\delta}}\left[\delta^2 - C(s-1)r^{2-m} \right] \dvol\\
& \leq  \int_{B_{2\eta}} u^s \frac{1}{\eta^2} \left(\frac{2\eta}{\e} \right)^{2\delta} \dvol + \left[\delta^2 - C(s-1)(2\eta)^{2-m} \right] \int_{B_\e \setminus B_{2\eta}} u^s \frac{r^{2\delta-2}}{\e^{2\delta}} \dvol\\
&\underset{Hölder}{\leq} 4 \e^{-2\delta} E^{\frac{p-s}{p}} (2\eta)^{h\frac{p-s}{p}+2\delta-2} \left(\int_{B_{2\eta}} u^p \dvol\right)^{\frac{s}{p}} \\
& \quad \quad \quad + \left[ \delta^2-(s-1) C (2\eta)^{2-m}\right] \int_{B_\e \setminus B_{2\eta}} u^s \frac{r^{2\delta-2}}{\e^{2\delta}}\dvol.
\end{align*}
Hence, recalling that the support of $u$ is contained in $B_R$,
\begin{equation}\label{App1p}
\begin{split}
(s-1)\int_{B_\e^c} u^s V  \dvol & \leq  4 \e^{-2\delta} E^{\frac{p-s}{p}} (2\eta)^{h\frac{p-s}{p}+2\delta-2} \left(\int_{B_{2\eta}} u^p \dvol\right)^{\frac{s}{p}}\\
& \quad \quad \quad + \left[ \delta^2-(s-1) C (2\eta)^{2-m}\right] \int_{B_\e \setminus B_{2\eta}} u^s \frac{r^{2\delta-2}}{\e^{2\delta}}\dvol
\end{split}
\end{equation}
for every $s\in (1,p]$. 
In our assumptions, we can choose $\delta$ and $s$ so that
\begin{align}\label{Cond:dp}
\delta^2-(s-1)C(2\eta)^{2-m}=0
\end{align}
and
\begin{align}\label{Cond:hp}
h \frac{p-s}{p}+2\delta-2\geq 0
\end{align}
for every $h$ satisfying \eqref{Cond:1p}. Indeed, following a case-by-case analysis:
\begin{itemize}
\item \underline{$m=2$ and $C=\frac{1}{p-1}$}: in this case we can just choose $s=p$ and $\delta=1$, so that \eqref{Cond:hp} is trivially satisfied for every $h\ge 0$.

\item \underline{$m=2$ and $C\in\left(0,\frac{1}{p-1}\right)$}: in this case we choose $\delta=\frac{pC}{h}$ and $s=1+\frac{\delta^2}{C}$.
Observing that
\begin{align*}
h \frac{p-s}{p}+2\delta-2\geq 0 \quad &\Leftrightarrow \quad h(p-s)+2p\delta-2p\geq 0 \\
&\Leftrightarrow \quad h\left(p-1-\frac{\delta^2}{C} \right)+2p\delta-2p\geq 0\\
&\Leftrightarrow \quad h^2(p-1)-h2p+p^2 C \geq 0,
\end{align*}
by the fact that $C<\frac{1}{p-1}$ it follows
\begin{align*}
&\Delta=4p^2-4p^2(p-1)C\geq 0\\
&\Rightarrow \quad h^2(p-1)-h2p+p^2 C \geq 0 \quad \forall h\geq \frac{p+p\sqrt{1-(p-1)C}}{p-1}\\
&\Rightarrow \quad h \frac{p-s}{p}+2\delta-2\geq 0\quad \forall h\geq \frac{p+p\sqrt{1-(p-1)C}}{p-1},
\end{align*}
implying \eqref{Cond:hp} when $\eta$ is small enough,.

\item \underline{$m=0$}: we choose $\delta=\frac{pC(2\eta)^{2}}{h}$ and $s=1+\frac{\delta^2}{C(2\eta)^{2}}$. As in previous case
\begin{align*}
h \frac{p-s}{p}+2\delta-2\geq 0 \quad &\Leftrightarrow \quad h(p-s)+2p\delta-2p\geq 0 \\
&\Leftrightarrow \quad  h^2(p-1)-h2p+p^2C(2\eta)^{2}\geq 0
\end{align*}
with
\begin{align*}
\Delta=4p^2-4p^2(p-1)C(2\eta)^{2}.
\end{align*}
Since we are interested in the limit as $\eta\to 0$, we get
\begin{align*}
h \frac{p-s}{p}+2\delta-2\geq 0 \quad \forall h\geq \frac{2p}{p-1}
\end{align*}
implying, again, \eqref{Cond:hp}.
\end{itemize}
From \eqref{Cond:dp} and \eqref{Cond:hp}, the inequality \eqref{App1p} implies
\begin{align*}
0\leq &(s-1)\int_{B_\e^c} u^s V  \dvol \leq \left(\int_{B_{2\eta}} u^p \dvol\right)^{\frac{s}{p}} 4\e^{-2\delta} E^{\frac{p-s}{p}} (2\eta)^{h\frac{p-s}{p}+2\delta-2}\xrightarrow[]{\eta\to 0}0.
\end{align*}
Since it holds for any fixed $\e>0$, we get
\begin{align*}
\int_M u^s V \dvol=0
\end{align*}
that, together with the fact that $V>0$ and $u\geq 0$, implies
\begin{align*}
u=v^-\equiv 0.
\end{align*}

\textbf{Step 2.} Now consider the general case where $v$ is not assumed to be compactly supported. Since $u:=v^- \in L^\infty\loc(M)$ by Proposition \ref{Prop:Caccioppoli1}, it follows that $\norm{u}{L^\infty(B_\e \setminus B_{2\eta})}<+\infty$. Consider the function
\begin{align*}
w:=\system{ll}{\left(\norm{u}{L^\infty(B_\e \setminus B_{2\eta})}-u\right)^- & \inn B_\e \\ 0 & \inn B_\e^c.}
\end{align*}
By Proposition \ref{Prop:BrezisKato},
\begin{align*}
(-\Delta+V)v\geq 0 \quad &\Rightarrow \quad (-\Delta+V)\left(\norm{u}{L^\infty(B_\e \setminus B_{2\eta})}-u\right) \geq 0 \\ &\Rightarrow \quad (-\Delta+V)w\geq 0,
\end{align*}
where the last inequality holds since $\left(\norm{u}{L^\infty(B_\e \setminus B_{2\eta})}-u\right)\geq 0$ in $B_\e \setminus B_{2\eta}$. Since $w\in L^p(M)$, by Step 1,	
\begin{align*}
\norm{u}{L^\infty(B_\e \setminus B_{2\eta})} \geq u \geq 0 \quad \inn B_\e.
\end{align*} 
In particular,
\begin{align}\label{Eq:Interpolation}
u\in L^p(B_\e)\cap L^\infty(B_\e) \quad \Rightarrow \quad u\in L^q(B_\e)\ \forall q\geq p.
\end{align}
As a consequence, by Proposition \ref{Prop:Caccioppoli1} applied to the test function $\varphi_{R,\e,\eta}$, for any $s\in (1,p]$
\begin{align*}
(s-1) & \int_{M} u^p V \varphi_{R,\e,\eta}^2 \dvol\\
& \leq (p-1) \int_{M} u^p V \varphi_{R,\e,\eta}^2 \dvol \\
& \leq \int_{M} u^p |\nabla \varphi_{R,\e,\eta}|^2 \dvol\\
& \leq \int_{B_{2\eta}} u^p \frac{1}{\eta^2} \left(\frac{2\eta}{\e} \right)^{2\delta} \dvol + \delta^2 \int_{B_\e \setminus B_{2\eta}} u^p \frac{r^{2\delta-2}}{\e^{2\delta}} \dvol + \int_{B_{R+\eta}\setminus B_R} u^p \frac{1}{\eta^2}\dvol
\end{align*}
and as $R\to +\infty$ we get
\begin{align*}
(s-1) & \int_{B_\e \setminus B_{2\eta}} u^p \frac{C}{r^m} \left( \frac{r}{\e}\right)^{2\delta} \dvol  + (s-1)\int_{B_\e^c} u^p V \dvol\\
	& \leq \lim_{R\to +\infty} (s-1) \int_M u^p V \varphi_{R,\e,\eta}^2 \dvol\\
	& \leq \int_{B_{2\eta}} u^p \frac{1}{\eta^2} \left(\frac{2\eta}{\e} \right)^{2\delta} \dvol + \delta^2 \int_{B_\e \setminus B_{2\eta}} u^p \frac{r^{2\delta-2}}{\e^{2\delta}} \dvol
\end{align*}
which implies
\begin{align*}
(s-1) & \int_{B_\e^c} u^p V \dvol\\
	& \leq \int_{B_{2\eta}} u^p \frac{1}{\eta^2} \left(\frac{2\eta}{\e} \right)^{2\delta} \dvol + \int_{B_\e \setminus B_{2\eta}}u^p  \frac{r^{2\delta-2}}{\e^{2\delta}} \left[\delta^2 - C(s-1) r^{2-m} \right] \dvol\\
& \leq \int_{B_{2\eta}} u^p \frac{1}{\eta^2} \left(\frac{2\eta}{\e} \right)^{2\delta} \dvol + \left[\delta^2 - C(s-1)(2\eta)^{2-m} \right] \int_{B_\e \setminus B_{2\eta}} u^p \frac{r^{2\delta-2}}{\e^{2\delta}} \dvol.
\end{align*}
In particular, this is equivalent to
\begin{align*}
(s-1)& \int_{B_\e^c} \left(u^\frac{p}{s}\right)^s V \dvol \\
& \leq \int_{B_{2\eta}} \left(u^\frac{p}{s}\right)^s \frac{1}{\eta^2} \left(\frac{2\eta}{\e} \right)^{2\delta} \dvol + \left[\delta^2 - C(s-1)(2\eta)^{2-m} \right] \int_{B_\e \setminus B_{2\eta}} \left(u^\frac{p}{s}\right)^s \frac{r^{2\delta-2}}{\e^{2\delta}} \dvol
\end{align*}
for any $s\in (1,p]$. Observing that $0\leq u^\frac{p}{s}\in L^p(B_\e)$ thanks to \eqref{Eq:Interpolation}, under the assumptions \eqref{Cond:2p} and \eqref{Cond:1p} we can apply the argument presented in previous step obtaining that $u\equiv 0$ in $B_\e^c$. By the arbitrariness of $\e>0$, we get $u\equiv 0$ and so $v$ is nonnegative.
\end{proof}

\section{Essential self-adjointness}\label{Sec:ESA}
As mentioned above, the positivity preserving property arises naturally when one deals with the self-adjointness of unbounded operators. In particular, as we are going to see, as soon as the $L^2$ positivity preserving property holds for a certain class of Schr\"odinger operators, then these operators turn out to be essentially self-adjoint.
\medskip

\subsection{Standard notions and results about self-adjointness} We recall some basic definitions about unbounded operators defined over Hilbert spaces. For further details, we refer to \cite{Ka,RS0,RS}.

Let $B$ be 
%
%
an Hilbert space with respect to the scalar product $(\cdot,\cdot)_B$, and $T:D(T)\subseteq B\to B$ an unbounded linear operator, where $D(T)$ is the domain of $T$. The \textit{adjoint of $T$}, denoted with $T^*$, is defined as the unbounded linear operator on $B$ whose domain is
\begin{align*}
D(T^*):=\bra{v\in B\ :\ \exists w\in B\ \textnormal{s.t.}\ (Tu,v)_B=(u,w)_B\ \forall u \in D(T)}
\end{align*}
and whose action is given by $T^*v=w$. In particular, by definition
\begin{align*}
(Tu,v)_B=(u,T^*v)_B \quad \forall u\in D(T), v\in D(T^*).
\end{align*}
The operator $T$ is said to be
\begin{itemize}
\item \textit{symmetric} if
\begin{align*}
(Tu,v)_B=(u,Tv)_B\quad \forall u,v\in D(T)
\end{align*}
or, equivalently, if $T\subseteq T^*$;
\item \textit{self-adjoint} if $T=T^*$, that is, if $T$ is symmetric and $D(T)=D(T^*)$;
\item \textit{essentially self-adjoint} if $T$ is symmetric and its closure $\overline{T}$ (defined as the operator whose graph is the closure of the graph of $T$) is self-adjoint.
\end{itemize}

\begin{remark}
We stress that
\begin{itemize}
\item by definition, the adjoint of an operator is a closed operator. In particular, if $T$ is symmetric (resp. self-adjoint), then $T$ is closable (resp. closed);
\item by an abstract fact (\cite[Theorem 5.29]{Ka}), $(T^*)^*=\overline{T}$;
\item a symmetric operator $T$ is essentially self-adjoint if and only if it has a unique self-adjoint extension (see \cite[page 256]{RS0}).
\end{itemize}
\end{remark}
\medskip

\subsection{Essential self-adjointness} A first application of Theorem \ref{Thm:Lp} to the theory of unbounded operators is the following result concerning the essential self-adjointness of $-\Delta+V$. The case $m=2$ and $C=1$ was previously obtained in \cite{MT} with a different approach, while the case $m=0$ is already contained in \cite{PVV}. Here we recover with a unified point of view both sets of assumptions, as well as all the new intermediate case $m=2$ and $C\in(0,1)$.

\begin{theorem}\label{Thm:ESA}
Let $(N,h)$ be a complete Riemannian manifold and define $M:=N\setminus K$, where $K\subset N$ is a compact subset. Consider $V\in L^\infty\loc (M)$ so that
\begin{align*}
V(x)\geq \frac{C}{r^m(x)}-B \quad in\ M,
\end{align*}
where $C\in [0,1]$, $m\in {0,2}$ and $B$ are positive constants and $r(x):=d^N(x,K)$ is the distance function from $K$.

If there exist two positive constants $E\geq 1$ and
\begin{align}\label{Cond:1}
h\geq\system{ll}{0 & if\ m=2\ and\ C=1\\ 2+2\sqrt{1-C} & if\ m=2\ and\ C\in (0,1)\\ 4 & if\ m=0}
\end{align}
so that
\begin{align}\label{Cond:2}
|B_r(K)|\leq E r^h \quad as\ r\to 0,
\end{align}
then the differential operator $-\Delta+V:C^\infty_c(M)\subset L^2(M) \to L^2(M)$ is essentially self-adjoint.
\end{theorem}

It is a standard fact (see \cite[Theorem X.26]{RS}) that a necessary and sufficient condition for the operator $-\Delta+V$ to be essentially self-adjoint on the domain $C^\infty_c(M)$ is that the unique solution $u\in L^2$ to $(-\Delta+V)u=0$ is the constant null function.

\begin{proof}
Let $\widetilde{V}=V+B>0$. Consider $u\in L^2$ a solution to $(-\Delta+\widetilde{V})u=0$: by Theorem \ref{Thm:Lp} applied both to $u$ and $-u$ it follows that $u=0$. This means that
\begin{align*}
(-\Delta+\widetilde{V})u=0 \quad \Rightarrow \quad u=0
\end{align*}
and hence $(-\Delta+\widetilde{V})$ is essentially self-adjoint on $C_c^\infty(M)$. By the invariance of the essential self-adjointness with respect to potential translations (see \cite[Proposition 4.1]{MT}), it follows that $(-\Delta+V)$ is essentially self-adjoint on $C_c^\infty(M)$, obtaining the claim.
\end{proof}

\begin{remark}

We stress that the bound $2+2\sqrt{1-C}$ is sharp. Namely, for $h=3$ and for every $n\ge 3$ and $C<1-(h-2)^2/4=3/4$ there exist a $C^2$ $n$-dimensional Riemannian manifold $N$ and a compact set $K\subset N$ 
such that 
\begin{itemize}
	\item $|B_r(K)|\leq E r^h$ for $r$ small enough and 
	\item the equation $(-\Delta +C/r^2)u=0$ admits an $L^2(M)$ solution, which in turn proves that $-\Delta+ \frac{C}{r^2}:C^\infty_c(M)\subset L^2(M) \to L^2(M)$ is not essentially self-adjoint.
\end{itemize}
Indeed, suppose first that $n=h=3$ and $C<3/4$. Let $N:=(\mathbb{R}_{\geq 0}\times_\sigma \mathbb{S}^2, \dint{r}+\sigma^2 g^{\mathbb{S}^2})$ be the model manifold with coordinates $(r,\theta)$ associated to the warping function
\begin{align*}
	\sigma(r):=r(1+r^2)(1+(2/b+1)r^2)^{-\frac{3}{2(2+b)}}, 
\end{align*}
where $b\in \left( 1,\frac{3}{2}\right)$ solves $C=b^2-b\in \left(0,\frac{3}{4}\right)$. Note that $\sigma'(0)=1$ and $\sigma(0)=\sigma''(0)=0$ so that $N$ is $C^2$. Let $K=\{0\}$ be the pole of the model manifold $N$ and define $M:=N\setminus K$ and $u:M\to \rr$ given by
\begin{align*}
	u(r,\theta):=\frac{1}{r^b (1+r^2)}.
\end{align*}
In particular, $u$ is a positive function satisfying
\begin{align*}
\left(-\Delta+\frac{C}{r^2}\right)u=0
\end{align*}
on $M$. Moreover $u\in L^2(M)$ since
\begin{align*}
\int_M u^2 \dvol & = b^{\frac{3}{(2+b)}} 4\pi \int_0^{+\infty} r^{2(1-b)} \left(b+(2+b)r^2 \right)^{-\frac{3}{2+b}} \dvol,
\end{align*}
which is integrable both around 0 and at $+\infty$ thanks to the choice of $b$.
Examples with $n>3=h$ can be obtained by considering $N^3\times \mathbb T^{n-3}$ where $N^3$ is as above, $\mathbb T^{n-3}$ is a $(n-3)$-dimensional torus, and $K=\{0\}\times \mathbb T^{n-3}$. We believe that similar counterexamples should exist also for non-integer $h\in(2,4)$, even if in that case we expect explicit computations to be much more tricky.

\end{remark}

\section{Operator core}\label{Sec:Core}
The second application of Theorem \ref{Thm:Lp} we present is the generalization of Theorem \ref{Thm:ESA} to the context of $L^p$ spaces with $p\neq 2$. Indeed, in this case a similar conclusion can be proved just replacing the self-adjointness with the property that $C_c^\infty$ is an operator core for $L^p$.

The general scheme we adopt will be summarized in the abstract result Theorem \ref{Thm:Final} at the end of this section. This is surely well-known to the experts, and can be deduced from a number of references quoted in the introduction of this paper. However we have not found it explicitly writen in the literature so that we decided to state it.
\medskip

\subsection{Standard notions and results about accretive operators}

We start by recalling the following definition.

\begin{definition}[Strongly continuous semigroup]
A family of bounded operators $\bra{T(t)}_{t\in \rr_{\geq 0}}$ defined over a Banach space $B$ is a \textnormal{strongly continuous semigroup} if
\begin{itemize}
\item $T(0)=I$;
\item $T(s)T(t)=T(s+t)$ for all $s,t\in \rr_{\geq 0}$;
\item for each $\psi\in B$ the map $t\mapsto T(t)\psi$ is continuous.
\end{itemize}
\end{definition}
\noindent
A special class of such semigroups is given by the contraction semigroups. A strongly continuous semigroup $\bra{T(t)}$ defined over a Banach space $B$ is said to be a \textit{contraction semigroup} if
\begin{align*}
||T(t)||\leq 1 \quad \forall t\in \rr_{\geq 0}.
\end{align*}
Here $\norm{\cdot}{}$ denotes the operator norm. The next proposition (\cite[Page 237]{RS}) shows that any contraction semigroup can be ``generated'' by a closed operator.

\begin{proposition}\label{Prop}
Let $T(t)$ be a strongly continuous semigroup on a Banach space $B$ and set
\begin{align*}
A_t:= t^{-1} (I-T(t))
\end{align*}
and
\begin{align*}
A:=\lim_{t\to 0} A_t
\end{align*}
defined over $D(A):=\bra{\psi\in B\ :\ \lim_{t\to 0} A_t \psi\ \textit{exists}}$. Then, $A$ is closed and densely defined.
\end{proposition}

\noindent The operator $A$ is called the \textit{infinitesimal generator} of $T(t)$. We will also say that $A$ generates $T(t)$ and write $T(t)=e^{-tA}$.

\bigskip
In the remaining part of this subsection, we introduce the notions of accretive and maximal accretive operators. To this aim, we recall that given a Banach space $(B,\norm{\cdot}{B})$ and $\psi\in B$, an element in its dual space $l\in B^*$ is said to be a \textit{normalized tangent functional} to $\psi$ if it satisfies
\begin{align*}
||l||_{B^\ast}=||\psi||_B \quad \andd \quad l(\psi)=||\psi||^2_B.
\end{align*}
Observe that by the Hahn-Banach theorem, each $\psi \in B$ has at least one normalized tangent functional.

\begin{definition}[Accretive and m-accretive operator]
A densely defined operator $A$ over a Banach space $B$ is said to be \textnormal{accretive} if for any $\psi\in D(A)$ there exists $l\in B^*$  a normalized tangent functional to $\psi$ so that $\textnormal{Re}(l(A\psi))\geq 0$. 

An accretive operator $A$ is said to be \textnormal{maximal accretive} (or \textnormal{m-accretive}) if it has no proper accretive extensions.
\end{definition}

\begin{remark}\label{Rmk:1}
We stress that
\begin{itemize}
\item every accretive operator is closable;
\item the closure of an accretive operator is again accretive.
\end{itemize}
As a consequence, every accretive operator has a smallest closed accretive extension. For a reference see \cite[Section X.8]{RS}.
\end{remark}

Now we can state the fundamental criterion.

\begin{theorem}[Fundamental criterion]\label{Thm:48}
A closed operator $A$ on a Banach space $B$ is the generator of a contraction semigroup if and only if $A$ is accretive and $\textnormal{Ran}(\lambda_0+A)=B$ for some $\lambda_0>0$.
\end{theorem}
\begin{proof}
We refer to \cite[Theorem X.48]{RS}.
\end{proof}


\begin{remark}\label{Rmk:Before49}
We stress that
\begin{enumerate}

\item by the Hille-Yosida theorem (\cite[Theorem X.47a]{RS}), if $A$ is the generator of a contraction semigroup, then the open half-line $(-\infty,0)$ is contained in the resolvent of $A$. In particular, it follows that $\textnormal{Ran}(I+A)=B$;

\item the generators of contraction semigroups are maximal accretive since the condition $\textnormal{Ran}(I+A)=B$ implies that $A$ has no proper accretive extensions. The converse ($A$ maximal accretive implies $A$ generates a contraction semigroup) holds if $B$ is an Hilbert space but not in the general Banach case. See \cite[Page 241]{RS}.
\end{enumerate}
\end{remark}


\medskip
\subsection{Operator core}

Let $V\in L^\infty\loc(M)$ and consider the differential operator $-\Delta+V$. If $p\in(1,+\infty)$, we define the operator $(-\Delta+V)_{p,max}$ associated to $-\Delta+V$ by the formula
\begin{align*}
(-\Delta+V)_{p,max} u = (-\Delta+V) u
\end{align*}
with domain
\begin{align*}
D\left((-\Delta+V)_{p,max}\right)=\bra{u\in L^p(M)\ :\ Vu\in L^1\loc(M),\ (-\Delta+V)u\in L^p(M)}.
\end{align*}
and the operator $(-\Delta+V)_{p,min}$ as
\begin{align*}
(-\Delta+V)_{p,min}:=(-\Delta+V)_{p,max}\Big|_{C_c^\infty(M)}.
\end{align*}
Observe that since $V\in L^p\loc(M)$, then $C_c^\infty(M)\subset D\left((-\Delta+V)_{p,max}\right)$ and hence the last definition makes sense.

\subsubsection{$\overline{(-\Delta+V)_{p,min}}$ is m-accretive}
Following the strategy of the proof adopted by O. Milatovic in \cite[Section 2]{Mi2}, the next step consists in proving that $\overline{(-\Delta+V)_{p,min}}$ is m-accretive. To this aim, we first prove that this operator is accretive.

\begin{lemma}\label{Lem:2.1}
Let $(M,g)$ be a (possibly incomplete) Riemannian manifold. Consider $0\leq V\in L^\infty\loc (M)$ and let $p\in (1,+\infty)$.

Then, the operator $\overline{(-\Delta+V)_{p,min}}$ is accretive.
\end{lemma}
\begin{proof}
It follows by Lemma 2.1 and Remark 2.2 in \cite{Mi2}. These latter are stated for complete manifolds, however the completeness assumption is not used, as remarked in the proof of \cite[Proposition 2.9 (b)]{Gu1}.
	
\end{proof}


From now on we consider a complete Riemannian manifold $(N,h)$ and define $M:=N\setminus K$, where $K\subset N$ is a compact subset. Let $V\in L^\infty\loc (M)$ so that
\begin{align*}
V(x)\geq \frac{C}{r^m(x)} \quad \inn M,
\end{align*}
where $C\in [0,1]$ and $m\in \{0,2\}$ are positive constants and $r(x):=d^N(x,K)$ is the distance function from $K$. Fix $p\in (1,+\infty)$ and suppose there exist two positive constants $E\geq 1$ and
\begin{align*}
h\geq\system{ll}{ 0 & \textnormal{if}\ m=2\ \textnormal{and}\ C=\frac{1}{p-1}\\  p+p\sqrt{1-\frac{C}{p-1}} & \textnormal{if}\ m=2\ \textnormal{and}\ C\in \left(0,\frac{1}{p-1}\right)\\  2p & \textnormal{if}\ m=0} \quad \textnormal{in case}\ p\geq 2
\end{align*}
or
\begin{align*}
h\geq \system{ll}{0 & \textnormal{if}\ m=2\ \textnormal{and}\ C=p-1\\ \frac{p+p\sqrt{1-(p-1)C}}{p-1} & \textnormal{if}\ m=2\ \textnormal{and}\ C\in \left(0,p-1\right)\\ \frac{2p}{p-1} & \textnormal{if}\ m=0} \quad \textnormal{in case}\ p< 2
\end{align*}
so that
\begin{align*}
|B_r(K)|\leq E r^h \quad \textnormal{as}\ r\to 0.
\end{align*}
In what follows we always assume to be in this setting.

\begin{remark}
We stress that in the present section we are requiring the validity of a condition stronger than the one of \eqref{Cond:1p} for the two indexes $p$ and $p'=p/(p-1)$ in order to obtain that both $\overline{(-\Delta+V)_{p,min}}$ and $\overline{(-\Delta+V)_{p',min}}$ are m-accretive. This latter will be used to ensure that the operator $(-\Delta+V)_{p,max}$ is accretive too.
\end{remark}

Thanks to the validity of Theorem \ref{Thm:Lp}, we are able to prove the next

\begin{theorem}\label{Thm:Accretive1}
$\overline{(-\Delta+V)_{p,min}}$ generates a contraction semigroup on $L^p(M)$. In particular, $\overline{(-\Delta+V)_{p,min}}$ is m-accretive.
\end{theorem}

The proof of Theorem \ref{Thm:Accretive1} can be obtained verbatim by the one of \cite[Theorem 1.3]{Mi2} just replacing Lemma 2.7 in \cite{Mi2} with Lemma \ref{Lem:2.7} below, which is a consequence of the validity of the positivity preserving property.

\begin{lemma}\label{Lem:2.7}
If $\lambda>0$, then $\textnormal{Ran}\left((-\Delta+V)_{p,min}+\lambda\right)$ is dense in $L^p(M)$.
\end{lemma}
\begin{proof}
Let $v\in L^{p'}(M)$ so that
\begin{align*}
\left\langle (\lambda+(-\Delta+V)_{p,min})u,v\right\rangle =0 \quad \forall u \in C^\infty_c(M),
\end{align*}
which is equivalent to the following distributional equality
\begin{align*}
(\lambda-\Delta+V)v=0.
\end{align*}
Since by hypothesis $V\in L^p\loc(M)$ and $v\in L^{p'}(M)$, by H\"older inequality $Vv\in L^1\loc$. Since $\Delta v=Vv+\lambda v$, we get $\Delta v \in L^1\loc(M)$. By Kato's inequality
\begin{align*}
-\Delta |v| \leq -\Delta v\ \textnormal{sign}\ v=(-\lambda v-Vv)\ \textnormal{sign}\ v\leq -V|v|
\end{align*}
and hence
\begin{align*}
(-\Delta+V)|v|\leq 0.
\end{align*}
By Theorem \ref{Thm:Lp} it follows that $|v|\leq 0$ and hence $v=0$.
\end{proof}

\subsubsection{$(-\Delta+V)_{p,max}$ is m-accretive}
After proving that $\overline{(-\Delta+V)_{p,min}}$ is m-accretive, the next stage is to show the same property for the operator $(-\Delta+V)_{p,max}$. We proceed by introducing the following result contained in \cite[Lemma I.25]{Gu2}

\begin{lemma}\label{Lem:I.25}
Let $p\in(1,+\infty)$ and $p'=p/(p-1)$. Then
\begin{align*}
(-\Delta+V)_{p,max}=\left((-\Delta+V)_{p',min}\right)^*.
\end{align*}
\end{lemma}

As a consequence, we get

\begin{theorem}\label{Thm:Accretive2}
$(-\Delta+V)_{p,max}$ generates a contraction semigroup on $L^p$. In particular, $(-\Delta+V)_{p,max}$ is m-accretive.
\end{theorem}
\begin{proof}
The proof follows as in \cite[Theorem 5]{GP}. Indeed, by Theorem \ref{Thm:Accretive1} the operator $\overline{(-\Delta+V)_{p',min}}$ generates a contraction semigroup and by Lemma \ref{Lem:I.25}
\begin{align*}
(-\Delta+V)_{p,max}=\left(\overline{(-\Delta+V)_{p',min}}\right)^*.
\end{align*}
Since adjoints of generators of contraction semigroups in reflexive Banach spaces again generate such semigroups \cite[p.138]{Ar}, it follows that  $(-\Delta+V)_{p,max}$ generates a contraction semigroup and thus is m-accretive.
\end{proof}

\subsubsection{Main result}

Before proceeding with the main result of this section, we recall the following

\begin{definition}
Let $T$ be a closed operator over a Banach space $B$. For any closable operator $S$ such that $\overline{S}=T$, its domain $D(S)$ is said to be a \textnormal{core} of $T$.
\end{definition}

\noindent In other words, $D\subset D(T)$ is a core of $T$ if the set $\bra{(u,Tu)\ :\ u\in D}$ is dense in $\Gamma(T)$.
 

\begin{theorem}\label{Thm:Core}
Let $(N,h)$ be a complete Riemannian manifold and define $M:=N\setminus K$, where $K\subset N$ is a compact subset. Consider $V\in L^\infty\loc (M)$ so that
\begin{align*}
V(x)\geq \frac{C}{r^m(x)}-B \quad in\ M,
\end{align*}
where $C\in [0,1]$, $m\in {0,2}$ and $B$ are positive constants and $r(x):=d^N(x,K)$ is the distance function from $K$, and fix $p\in (1,+\infty)$.

If there exist two positive constants $E\geq 1$ and
\begin{align}\label{Cond:3p1}
h\geq\system{ll}{ 0 & if\ m=2\ and\ C=\frac{1}{p-1}\\  p+p\sqrt{1-\frac{C}{p-1}} & if\ m=2\ and\ C\in \left(0,\frac{1}{p-1}\right)\\  2p & if\ m=0} \quad in\ case\ p\geq 2
\end{align}
or
\begin{align}\label{Cond:3p2}
h\geq\system{ll}{0 & if\ m=2\ and\ C=p-1\\ \frac{p+p\sqrt{1-(p-1)C}}{p-1} & if\ m=2\ and\ C\in \left(0,p-1\right)\\ \frac{2p}{p-1} & if\ m=0} \quad in\ case\ p< 2
\end{align}
so that
\begin{align}\label{Cond:4p}
|B_r(K)|\leq E r^h \quad as\ r\to 0,
\end{align}
then $C^\infty_c(M)$ is an operator core for $(-\Delta+V)_{p,max}$.
\end{theorem}

%
\begin{proof}
Let $\widetilde{V}=V+B>0$. By Theorem \ref{Thm:Accretive1} and Theorem \ref{Thm:Accretive2}, both $\overline{(-\Delta+\widetilde{V})_{p,min}}$ and $(-\Delta+\widetilde{V})_{p,max}$ are m-accretive. By the fact that $\overline{(-\Delta+\widetilde{V})_{p,min}}\subset (-\Delta+\widetilde{V})_{p,max}$ and by the definition of m-accretive operator, it follows that $\overline{(-\Delta+\widetilde{V})_{p,min}}=(-\Delta+\widetilde{V})_{p,max}$, obtaining that $C^\infty_c(M)$ is an operator core for $(-\Delta+\widetilde{V})_{p,max}$. By the invariance of this property with respect to potential translations (see Remark \ref{Rmk:CoreInvariance} below), we get the claim.
\end{proof}

\begin{remark}\label{Rmk:CoreInvariance}
We observe that $C^\infty_c(M)$ is an operator core for $(-\Delta+V)_{p,max}$, then $C^\infty_c$ is an operator core also for $(-\Delta+V+\lambda)_{p,max}$ for every $\lambda \in \rr$.

Indeed, suppose that $C^\infty_c(M)$ is an operator core for $(-\Delta+V)_{p,max}$, meaning that $\{(u,(-\Delta+V)u)\ :\ u\in C^\infty_c(M)\}$ is dense in $\Gamma((-\Delta+V)_{p,max})$. Fixed $\lambda\in \rr$, consider $(u,(-\Delta+V+\lambda)u)\in \Gamma((-\Delta+V+\lambda)_{p,max})$ and observe that
\begin{align*}
D((-\Delta+V+\lambda)_{p,max})=D((-\Delta+V)_{p,max})
\end{align*}
and hence
\begin{align*}
(u,(-\Delta+V)u)\in \Gamma((-\Delta+V)_{p,max}).
\end{align*}
By the fact that $C^\infty_c(M)$ is an operator core for $(-\Delta+V)_{p,max}$ it follows that there exists $\{u_n\}_n\subset C^\infty_c(M)$ so that
\begin{align*}
(u_n,(-\Delta+V)u_n)\xrightarrow[]{n}(u,(-\Delta+V)u) \quad \inn \Gamma((-\Delta+V)_{p,max}),
\end{align*}
i.e.
\begin{align*}
\norm{u_n-u}{L^p}+\norm{(-\Delta+V)(u_n-u)}{L^p}\xrightarrow[]{n}0,
\end{align*}
implying that
\begin{enumerate}
\item $\norm{u_n-u}{L^p}\xrightarrow[]{n}0$
\item $\norm{(-\Delta+V)(u_n-u)}{L^p}\xrightarrow[]{n}0$.
\end{enumerate}
Whence, by Minkowski inequality,
\begin{align*}
\norm{(-\Delta+V+\lambda)(u_n-u)}{L^p} & \leq \norm{(-\Delta+V)(u_n-u)}{L^p}+|\lambda| \norm{u_n-u}{L^p}\xrightarrow[]{n}0.
\end{align*}
and hence $(-\Delta+V+\lambda)u_n\xrightarrow[]{L^p}(-\Delta+V+\lambda)u$. So
\begin{align*}
(u_n,(-\Delta+V+\lambda)u_n)\xrightarrow[]{n}(u,(-\Delta+V+\lambda)u) \quad \inn \Gamma((-\Delta+V+\lambda)_{p,max}).
\end{align*}
It follows that for every $\lambda\in \rr$ the set $\{(u,(-\Delta+V+\lambda)u\ :\ u\in C^\infty_c(M)\}$ is dense in $\Gamma((-\Delta+V+\lambda)_{p,max})$ and hence $C^\infty_c(M)$ is an operator core for $(-\Delta+V+\lambda)_{p,max}$. 
\end{remark}

\begin{remark}
In case $p=2$ (and hence $p'=2$), we recover the result contained in Theorem \ref{Thm:ESA}. Indeed, under the assumptions of Theorem \ref{Thm:ESA}, the condition 
\begin{center}
$C^\infty_c(M)$ \textit{is an operator core for} $(-\Delta+V)_{2,max}$
\end{center}
means exactly that the operator $-\Delta+V$ is essentially self-adjoint on $C^\infty_c(M)$. 
\end{remark}
\medskip

\subsection{Consequence of the above construction} As we can see from the previous discussion, the construction carried out in this section is guaranteed even under more general assumptions than those required in Theorem \ref{Thm:Core}. In fact, we can observe that for the proofs of Theorems \ref{Thm:Accretive1} and \ref{Thm:Accretive2}, which are the key results from which Theorem \ref{Thm:Core} immediately follows, only the property of positivity preservation for the operator $-\Delta + V$ is required. As a direct consequence of this fact, we obtain a machinery that ensures that $C^\infty_c$ is an operator core for the $p$-maximal extension of a given Schrödinger operator as soon as the underlying manifold satisfies the positivity preservation for that operator for the index $p$ and for its dual $p'$. We summarize this result in the following

\begin{theorem}\label{Thm:Final}
Let $(M,g)$ be a (possibly) incomplete Riemannian manifold. Consider $0<V\in L^\infty\loc(M)$ and $p\in (1,+\infty)$ and define $p'=\frac{p}{p-1}$.

If $(M,g)$ satisfies both the $L^p$ and $L^{p'}$ positivity preserving property for the operator $-\Delta+V$, then $C^\infty_c(M)$ is an operator core for $(-\Delta+V)_{p,max}$.
\end{theorem}

\bibliography{references}

\begin{thebibliography}{10}

\bibitem{ACR}
L.~J. Alias, G.~Colombo, and M.~Rigoli.
\newblock Growth of subsolutions of ${\Delta}_p u= {V}|u|^{p-2} u $ and of a
  general class of quasilinear equations.
\newblock {\em arXiv preprint arXiv:2304.05829}, 2023.

\bibitem{Ar}
W.~Arendt, C.~J.~K. Batty, M.~Hieber, and F.~Neubrander.
\newblock {\em Vector-valued {L}aplace transforms and {C}auchy problems},
  volume~96 of {\em Monographs in Mathematics}.
\newblock Birkh\"{a}user/Springer Basel AG, Basel, second edition, 2011.

\bibitem{BS}
D.~Bianchi and A.~G. Setti.
\newblock Laplacian cut-offs, porous and fast diffusion on manifolds and other
  applications.
\newblock {\em Calculus of Variations and Partial Differential Equations},
  57(1):4, 2018.

\bibitem{BFP}
A.~Bisterzo, A.~Farina, and S.~Pigola.
\newblock ${L}^p_{loc}$ positivity preservation and {L}iouville-type theorems.
\newblock {\em arXiv preprint arXiv:2304.00745}, 2023.

\bibitem{BM}
A.~Bisterzo and L.~Marini.
\newblock The ${L}^\infty$-positivity preserving property and stochastic
  completeness.
\newblock {\em Potential Analysis}, pages 1--18, 2022.

\bibitem{BMS}
M.~Braverman, O.~Milatovic, and M.~Shubin.
\newblock Essential self-adjointness of {S}chr{\"o}dinger-type operators on
  manifolds.
\newblock {\em Russian Mathematical Surveys}, 57(4):641, 2002.

\bibitem{Br}
H.~Brezis.
\newblock Semilinear equations in $\mathbb{R}^{N}$ without condition at
  infinity.
\newblock {\em Applied Mathematics and Optimization}, 12:271--282, 1984.

\bibitem{CdV}
Y.~Colin~de Verdi\`ere.
\newblock {P}seudo-laplaciens. {I}.
\newblock {\em Ann. Inst. Fourier (Grenoble)}, 32(3):xiii, 275--286, 1982.

\bibitem{Gu1}
B.~G{\"u}neysu.
\newblock Sequences of {L}aplacian cut-off functions.
\newblock {\em The Journal of Geometric Analysis}, 26(1):171--184, 2016.

\bibitem{Gu3}
B.~G{\"u}neysu.
\newblock The {BMS} conjecture.
\newblock In {\em Ulmer Seminare}, volume~20, pages 97--101, 2017.
\newblock Preprint available at {https://arxiv.org/pdf/1709.07463.pdf}.

\bibitem{Gu2}
B.~G{\"u}neysu et~al.
\newblock {\em Covariant {S}chr{\"o}dinger semigroups on {R}iemannian
  manifolds}.
\newblock Springer, 2017.

\bibitem{GP}
B.~G{\"u}neysu and S.~Pigola.
\newblock ${L}^p$-interpolation inequalities and global {S}obolev regularity
  results (with an appendix by {O}gnjen {M}ilatovic).
\newblock {\em Annali di Matematica Pura ed Applicata (1923-)}, 198:83--96,
  2019.

\bibitem{GPSV}
B.~G{\"u}neysu, S.~Pigola, P.~Stollmann, and G.~Veronelli.
\newblock A new notion of subharmonicity on locally smoothing spaces, and a
  conjecture by {B}raverman, {M}ilatovic, {S}hubin.
\newblock {\em arXiv preprint arXiv:2302.09423}, 2023.

\bibitem{GPo}
B.~G{\"u}neysu and O.~Post.
\newblock Path integrals and the essential self-adjointness of differential
  operators on noncompact manifolds.
\newblock {\em Mathematische Zeitschrift}, 275(1-2):331--348, 2013.

\bibitem{HMS}
M.~Hinz, J.~Masamune, and K.~Suzuki.
\newblock Removable sets and ${L}^p$-uniqueness on manifolds and metric measure
  spaces.
\newblock {\em Nonlinear Analysis}, 234:113296, 2023.

\bibitem{Ka0}
T.~Kato.
\newblock Schr{\"o}dinger operators with singular potentials.
\newblock {\em Israel Journal of Mathematics}, 13:135--148, 1972.

\bibitem{Ka2}
T.~Kato.
\newblock Lp-theory of schr{\"o}dinger operators with a singular potential.
\newblock In {\em North-Holland Mathematics Studies}, volume 122, pages 63--78.
  Elsevier, 1986.

\bibitem{Ka}
T.~Kato.
\newblock {\em Perturbation theory for linear operators}, volume 132.
\newblock Springer Science \& Business Media, 2013.

\bibitem{MV}
L.~Marini and G.~Veronelli.
\newblock Some functional properties on {C}artan-{H}adamard manifolds of very
  negative curvature.
\newblock {\em arXiv preprint arXiv:2105.09024}, 2021.

\bibitem{Ma}
J.~Masamune.
\newblock Essential self-adjointness of {L}aplacians on {R}iemannian manifolds
  with fractal boundary.
\newblock {\em Comm. Partial Differential Equations}, 24(3-4):749--757, 1999.

\bibitem{Mi1}
O.~Milatovic.
\newblock On m-accretive {S}chr{\"o}dinger operators in ${L}^p$-spaces on
  manifolds of bounded geometry.
\newblock {\em Journal of mathematical analysis and applications},
  324(2):762--772, 2006.

\bibitem{Mi2}
O.~Milatovic.
\newblock On m-accretivity of perturbed {B}ochner {L}aplacian in ${L}^p$ spaces
  on {R}iemannian manifolds.
\newblock {\em Integral Equations and Operator Theory}, 68(2):243--254, 2010.

\bibitem{MT}
O.~Milatovic and F.~Truc.
\newblock Self-adjoint extensions of differential operators on {R}iemannian
  manifolds.
\newblock {\em Annals of Global Analysis and Geometry}, 49:87--103, 2016.

\bibitem{PVV}
S.~Pigola, D.~Valtorta, and G.~Veronelli.
\newblock Approximation, regularity and positivity preservation on {R}iemannian
  manifolds.
\newblock {\em arXiv preprint arXiv:2301.05159}, 2023.

\bibitem{Po}
A.~C. Ponce.
\newblock Elliptic pdes, measures and capacities.
\newblock {\em Tracts in Mathematics}, 23:10, 2016.

\bibitem{RS0}
M.~Reed and B.~Simon.
\newblock {\em Methods of {M}odern {M}athematical {P}hysics: {F}unctional
  {A}nalysis; {R}ev. ed}.
\newblock Academic press, 1980.

\bibitem{RS}
M.~Reed, B.~Simon, and S.~Reed.
\newblock {\em Methods of {M}odern {M}athematical {P}hysics: {F}ourier
  {A}nalysis, {S}elf-{A}djointness}.
\newblock 1975.

\bibitem{Sh}
M.~Shubin.
\newblock Essential self-adjointness for semi-bounded magnetic
  {S}chr{\"o}dinger operators on non-compact manifolds.
\newblock {\em Journal of Functional Analysis}, 186(1):92--116, 2001.

\bibitem{Sh0}
M.~A. Shubin.
\newblock Spectral theory of elliptic operators on noncompact manifolds.
\newblock {\em Ast{\'e}risque}, 207(5):35--108, 1992.

\end{thebibliography}
\bibliographystyle{abbrv}

\end{document}